\documentclass{amsart}
\usepackage{graphicx}
\usepackage[mathscr]{eucal}
\usepackage{amssymb}
\newtheorem{thm}{Theorem}

\newtheorem{lem}[thm]{Lemma}
\newtheorem{prop}[thm]{Proposition}
\theoremstyle{definition}

\theoremstyle{remark}

\newtheorem{qu}[thm]{Question}

\newcommand{\E}{\mathcal{E}}

\newcommand{\mset}{\emptyset}

\newcommand{\C}{\mathcal{C}}
\begin{document}
\baselineskip=18pt
\title{Properties of the space of sections of some Banach bundles}
\author{Aldo J. Lazar}
\address{School of Mathematical Sciences\\
         Tel Aviv University\\
         Tel Aviv 69978, Israel}
\email{aldo@post.tau.ac.il}
\date{February 5, 2018}


\commby{}
\begin{abstract}

One shows for Banach bundles in a certain class that having a second countable locally compact Hausdorff base space and separable fibers implies the separability of the Banach space of the all sections that vanish at infinity. In the reverse direction, it is proved that for a Banach bundle with locally compact Hausdorff base space the separability of the space of all the sections that vanish at infinity implies that the base space is second countable. If the base space is compact and the space of all the sections of the bundle is generated by a weakly compact subset then the base space is an Eberlein compact.

\end{abstract}
\maketitle
\section{Introduction}

   It is well known that a compact Hausdorff space $T$ is second countable if and only if the Banach space $C(T)$ consisting of all the scalar valued functions on $T$ is separable. Similarly, the compact Hausdorff $T$ is an Eberlein compact if and only if $C(T)$ is generated by a weakly compact subset. We generalize to some extent these results to Banach bundles with compact and locally compact Hausdorff base spaces. To establish the separability of the space of sections we restrict the discussion to locally uniform Banach bundles, a class of Banach bundles that was introduced in \cite{L}; its definition will be given below. Namely, we prove that the space of all the sections that vanish at infinity is separable provided that the Banach bundle is locally uniform, the base space is a second countable locally compact Hausdorff space, and all the fiber spaces of the bundle are separable. In the other direction we prove that if the Banach space of all the sections of the Banach bundle $\xi := (\E,p,T)$ that vanish at infinity on the locally compact Hausdorff space $T$ is separable then $T$ is second countable; of course, each fiber is a separable Banach space. We also show that if the base space is compact Hausdorff and the space of all the sections of $\xi$ is weakly compact generated (WCG) then $T$ is Eberlein compact. For other separability results one should consult \cite[215 ff]{G} and the references there.

   We discuss here Banach bundles for which the norm is upper semi-continuous that is $(H)$ Banach bundles in the terminology of \cite{DG}. We shall always suppose that the base space of a Banach bundle is Hausdorff. For a Banach bundle $\xi := (\E,p,T)$ we denote by $\Gamma(\xi)$ the space of all the sections of $\xi$; if the base space is compact $\Gamma(\xi)$ is a Banach space when one endows it with the sup norm. Similarly, if $T$ is locally compact then $\Gamma_0(\xi)$ stands for the space of all the sections of $\xi$ that vanish at infinity; we endow it with the sup norm to make it a Banach space.

   We shall denote the closed unit ball of the Banach space $Y$ by $Y_1$. Let $X$ be a Banach space with its family of bounded closed subsets endowed with the Hausdorff metric, $T$ a Hausdorff topological space and $t\to X(t), t\in T$ a map such that $t\to X(t)_1$ is continuous. Set $\E := \cup_{t\in T} (\{t\}\times X(t))$ and define $p :\E\to T$ by $p((t,x)) := t$. Given an open subset $V$ of $T$ and an open subset $O$ of $X$ we denote $\mathcal{U}(V,O) := \{(t,x)\mid t\in V,\ x\in O\cap X(t)\}$. It is shown in \cite[section 5]{L} that the family of all the sets $\{\mathcal{U}(V,O)\}$ is the base of a topology on $\E$ such that $(\E,p,T)$ is a Banach bundle for which the norm is continuous on the bundle space $\E$. A Banach bundle isomorphic to a Banach bundle as detailed above is called a uniform Banach bundle. A Banach bundle $\xi := (\E,p,T)$ is called a locally uniform Banach bundle if there is a family $\{T_{\iota}\}$ of closed subsets of $T$ such that $\{\mathrm{Int}(T_{\iota})\}$ is an open cover of $T$ and every  restriction $\xi\mid T_{\iota}$ is a uniform Banach bundle.

   An Eberlein compact is a compact Hausdorff space that is homeomorphic to a weakly compact subset of a Banach space. A Banach space $X$ is called weakly compact generated (WCG) if it contains a weakly compact subset whose linear span is dense in $X$. A thorough discussion of Eberlein compact spaces and of WCG Banach spaces can be found in \cite{Z}.

\section{results}

We begin by discussing conditions on the Banach bundle that ensure the separability of the space of sections.

\begin{prop} \label{P:compact}

   Let $\xi := (\E,p,T)$ be a locally uniform Banach bundle with $T$ compact and metrizable and $p^{-1}(t)$ separable for each $t\in T$. Then $\Gamma(\xi)$ is separable.

\end{prop}

The proof uses a lemma that undoubtedly is known but a proof is included for lack of a reference. The proof presented here relies on a rather advanced result but the lemma can be given a direct elementary proof.

\begin{lem} \label{L:separable}

  Let $X$ be a Banach space and $Y$ a closed subspace of $X$. If $Y$ and $X/Y$ are separable then $X$ is separable.

\end{lem}

\begin{proof}

   Let $\theta : X\to X/Y$ be the quotient map. By \cite{BG} there is a continuous map $\psi : X/Y\to X$ such that $\theta(\psi(z)) = z$ for every $z\in X/Y$. It is immediately seen that the map $x\to (x - \psi(\theta(x)),\theta(x))$ is a homeomorphism of $X$ onto the separable space $Y\times X/Y$; the inverse map is $(y,z)\to y + \psi(z)$.

\end{proof}

\begin{proof}[Proof of Proposition \ref{P:compact}]

   We begin with a particular case: let $\xi$ be a uniform Banach bundle. Thus we may consider the following setting: each fiber $p^{-1}(t)$ is a subspace of some Banach space $X$ such that the map $t\to p^{-1}(t)$ from $T$ to the family of all the closed subsets of $X$ has the property that $t\to p^{-1}(t)_1$ is continuous when the collection of all the closed bounded subsets of $X$ is endowed with the Hausdorff metric. We claim that the closed subspace $\tilde{X}$ of $X$ generated by $\cup \{p^{-1}(t) \mid t\in T\}$ is separable. Indeed, with $\{t_n\}$ a countable dense subset of $T$ and $\{x_m^n\}_{m=1}^{\infty}$ a dense subset of $p^{-1}(t_n)_1$, it is easily seen that the set of all linear combinations with rational coefficients of $\{x_n^m \mid n, m\in \mathbb{N}\}$ is dense in $\tilde{X}$. From now on to simplify the notation we shall assume that $\tilde{X} = X$. In this situation the Banach space $\C(T,X) = \C(T)\otimes X$ is separable hence its subspace $\Gamma(\xi)$ is separable too.

   We consider now the general case. Thus, there is a family $\{T_i\}_{i=1}^n$ of closed subsets of $T$ such that $\{\mathrm{Int}(T_i)\}_{i=1}^n$ is a cover of $T$, $\mathrm{Int}(T_i)\neq \mset$ for each $i$, and the restriction of $\xi$ to each $T_i$ is a uniform Banach bundle. We shall prove the conclusion by induction. The step $n=1$ was treated above and we shall suppose now that if the base space can be covered by $n-1$ closed sets as before-mentioned then the conclusion of the theorem is valid. Set $S := \cup_{i=1}^{n-1} T_i$. The restriction map $\varphi\to \varphi\mid S$ maps $\Gamma(\xi)$ onto the separable space $\Gamma(\xi\mid S)$ by \cite[p. 15]{DG}. The kernel of this map $\{\varphi\in \Gamma(\xi)\mid \varphi\mid S \equiv 0\}$ which can be considered as a closed subspace of the separable space $\Gamma(\xi\mid T_n)\}$. Thus $\Gamma(\xi)$ is separable by Lemma \ref{L:separable}.

\end{proof}

In the above proposition we required the bundle to be locally uniform. If one discards this hypothesis it may happen that $\Gamma(\xi)$ is not separable even if the base space is metrizable and all the fibers are separable Banach spaces. Such a situation occurs in \cite[Example 5.1]{M}. We consider now a setting more general than in Proposition \ref{P:compact}.

\begin{thm} \label{T:l.c.}

   Let $\xi := (\E,p,T)$ be a locally uniform Banach bundle. If $T$ is a second countable locally compact space and each fiber $p^{-1}(t),\ t\in T,$ is separable then $\Gamma_0(\xi)$ is separable.

\end{thm}

\begin{proof}

   There is a sequence of compact subsets $\{T_n\}_{n=1}^{\infty}$ of $T$ such that $\rm{Int}{T_1}\neq \mset$, $T_n\subseteq \rm{Int}(T_{n+1})$, $n = 1,2,\ldots $, and $T = \cup_{n=1}^{\infty} T_n$, see \cite[Theorem XI.7.2]{D}. Let $f_n$ be a real valued continuous function on $T$ such that $f_n\mid T_n \equiv 1$, $0\leq f_n\leq 1$, and $f_n(t) = 0$ if $t\in T\setminus \rm{Int}(T_{n+1})$, $n = 1,2,\ldots$ .For each $n$ the Banach space $\Gamma(\xi\mid T_n)$ is separable by Proposition \ref{P:compact} and we choose a dense countable set $\{\varphi_m^n\}_{m=1}^{\infty}$ in this space. By using again \cite[p. 15]{DG} we extend each $\varphi_m^n$ to a section of $\xi$ which, for simplicity, we shall denote also by $\varphi_m^n$ with the requirement that
    $$
     \|\varphi_m^n(t)\|\leq \|\varphi_m^n\mid T_n\|
    $$
   for each $t\in T$. Denote now $\psi_m^n := f_n\varphi_m^{n+1}$. We claim that $\{\psi_m^n\}_{m,n=1}^{\infty}$ is dense in $\Gamma_0(\xi)$.

   To prove the claim let $\varphi\in \Gamma_0(\xi)$, $\epsilon > 0$ and $n\in \mathbb{N}$ such that
    $$
     \{t\in T\mid \|\varphi(t)\| \geq \epsilon\}\subseteq T_n.
    $$
   Choose $m$ such that $\|\varphi(t) - \varphi_m^{n+1}(t)\| < \epsilon$ for every $t\in T_{n+1}$. For this choice of $m$ we have
    $$
     \|\varphi(t) - \psi_m^n(t)\| = \|\varphi(t) - \varphi_m^{n+1}(t)\| < \epsilon, \ t\in T_n,
    $$
   and
    $$
     \|\varphi_m^{n+1}(t)\|\leq \epsilon + \|\varphi(t)\| < 2\epsilon, \ t\in \rm{Int}(T_{n+1})\setminus T_n,
    $$
   hence
    $$
     \|\varphi(t) - \psi_m^n(t)\| = \|\varphi(t) - f_n(t)\varphi_m^{n+1}(t)\|\leq \|\varphi(t)\| + |f_n(t)|\|\varphi_m^{n+1}(t)\| < 3\epsilon, \ t\in \rm{Int}(T_{n+1})\setminus T_n,
    $$
   and $\|\varphi(t) - \psi_m^n(t)\| = \|\varphi(t)\| < \epsilon$ if $t\in T\setminus \rm{Int}(T_{n+1})$. We got $\|\varphi - \psi_m^n\|\leq 3\epsilon$ and this ends the proof.

\end{proof}

\begin{qu}

   Let $\xi := (\E,p,T)$ be a Banach bundle with $T$ an Eberlein compact. Suppose moreover that each fiber $p^{-1}(t)$ is WCG. Is $\Gamma(\xi)$ WCG? What if $\xi$ is locally uniform?

\end{qu}

We now discuss how certain properties of the space of sections are reflected in properties of the base space.

\begin{thm} \label{T:c}

   Let $\xi := (\E,p,T)$ be a Banach bundle with $T$ compact Hausdorff. If $\Gamma(\xi)$ is separable (WCG) then $T$ is second countable (Eberlein compact, respectively) and each fiber space of $\xi$ is separable (WCG, respectively).

\end {thm}

\begin{proof}

   We begin by showing that each point $t$ of the base space has a compact neighbourhood $U_t$ such that the Banach space $C(U_t)$ is isomorphic with a subspace of $\Gamma(\xi\mid U_t)$. By \cite[p. 15]{DG} there is a section $\varphi_t$ of $\Gamma(\xi)$ such that $\|\varphi_t(t)\| = 1$. From the upper semicontinuity on $\E$ of the norm it follows that $\{s\in T\mid \|\varphi_t(t) - \varphi_t(s)\| < 1/2\}$ is a neighbourhood $V_t$ of $t$. For each $s\in V_t$ we have $\|\varphi_t(s)\| > 1 - 1/2 = 1/2$. Let $U_t$ be a compact neighbourhood of $t$ contained in $V_t$. The map $f\to f\cdot \varphi_t\mid U_t$ is an isomorphism of $C(U_t)$ into $\Gamma(\xi\mid U_t)$. Indeed,
   $$
    \frac{1}{2}\|f\|\leq \|f\cdot \varphi_t\mid U_t\|\leq \|f\|\|\varphi_t\|.
   $$
   We shall denote by $\Phi_t$ this map of $C(U_t)$ into $\Gamma(\xi\mid U_t)$

   Now, if $\Gamma(\xi)$ is separable each $\Gamma(\xi\mid U_t)$ is separable hence each $C(U_t)$ is separable. Thus each point $t\in T$ has a second countable neighbourhood and we infer that $T$ is second countable.

   If $\Gamma(\xi)$ is WCG then each $\Gamma(\xi\mid U_t)$ is WCG by being a quotient of $\Gamma(\xi)$. Therefore all the norm closed balls of $\Phi_t(C(U_t))^*$, $t\in T$,  are Eberlein compacts in their weak star topology by \cite[Theorem 4.9]{Z}. Hence the image $B_t$ by $(\Phi_t^*)^{-1}$ of the norm closed unit ball of $C(U_t)^*$ is an Eberlein compact in its weak star topology since it is a weak star closed subset of a norm closed ball centered at the origin of $\Phi_t(C(U_t))^*$. Since $U_t$ is homeomorphic with a weak star closed subset of $B_t$ we infer that $U_t$ is an Eberlein compact. It turns out that $T$ is a finite union ot Eberlein compact. The disjoint union $S$ of these finitely many Eberlein compacts is obviously an Eberlein compact. There is a natural continuous map of $S$ onto $T$ hence $T$ is an Eberlein compact by \cite[Theorem 2.15]{Z}.

\end{proof}

For locally compact base spaces we have the following result.

\begin{thm}

   If $\xi := (\E,p,T)$ is a Banach bundle with a locally compact base space $T$ and $\Gamma_0(\xi)$ is separable then $T$ is second countable and each fiber space of $\xi$ is separable.

\end{thm}

\begin{proof}

   Let $\{\varphi_n\}_{n=1}^{\infty}$ be a dense countable subset of $\Gamma_0(\xi)$. For each $n$ the set $\{t\in T\mid \|\varphi_n(t)\|\}$ is $\sigma$-compact and $T$ is the union of all these sets. As in the proof of Theorem \ref{T:l.c.} there is an increasing sequence of compact subsets $\{T_m\}$ such that $T_m\subset \rm{Int}(T_{m+1})$ and $T = \cup_{m=1}^{\infty} T_m$. Each $T_m$ is second countable by Theorem \ref{T:c} hence $T$ is second countable. The second assertion is obvious.

\end{proof}

\bibliographystyle{amsplain}

\begin{thebibliography}{99}

 \bibitem{BG}

  R. G. Bartle and L. M. Graves, \emph{Mappings between function spaces}, Trans. Amer. Math. Soc. \textbf{72} (1952), 400--413.

 \bibitem{D}

  J. Dugundji, \emph{Topology}, Allyn and Bacon, Boston, 1967.

 \bibitem{DG}

  M. J. Dupr\'{e} and R. M. Gillette, \emph{Banach bundles, Banach modules and automorphisms of $C^*$-algebras}, Pitman Pub. Co., New York, 1983.

 \bibitem{G}

  G. Gierz, \emph{Bundles of topological vector spaces and their duality}, Springer-Verlag, Berlin, 1982.

 \bibitem{L}

  A.J. Lazar, \emph{A selection theorem for Banach bundles and applications}, to appear in J. Math. Anal. Appl.

 \bibitem{M}

  O. Mar\'{e}chal, \emph{Champs mesurables d'espaces hilbertiens}, Bull. Sci. Math. 93 (1969), 113--143.

 \bibitem{Z}

  V. Zizler, \emph{Nonseparable Banach spaces} in \emph{Handbook of the geometry of Banach spaces}, vol. 2, ed. W.B. Johnson and J. Lindenstrauss, Elsevier, Amsterdam, 2003.

\end{thebibliography}

\end{document}